\documentclass[11pt,a4paper]{amsart}

\usepackage[T1]{fontenc}
\usepackage{lmodern}
\usepackage{amsmath,amssymb,mathtools}
\usepackage{microtype}
\usepackage[hidelinks]{hyperref}
\hypersetup{
 pdftitle={A Logarithm-Free Critical Endpoint Estimate for the Two-Dimensional Hermite Operator},
 pdfauthor={}
}

\allowdisplaybreaks
\newtheorem{theorem}{Theorem}[section]
\newtheorem{proposition}[theorem]{Proposition}
\newtheorem{lemma}[theorem]{Lemma}
\numberwithin{equation}{section}

\newcommand{\R}{\mathbb R}
\newcommand{\T}{\mathbb T}
\newcommand{\N}{\mathbb N}
\newcommand{\calH}{\mathcal H}
\newcommand{\one}{\mathbf 1}
\newcommand{\dd}{\,\mathrm d}
\newcommand{\norm}[1]{\lVert #1\rVert}
\makeatletter
\newcommand{\proofstep}[1]{%
  \par\begingroup
  \dimen@=4\baselineskip
  \vskip\z@\@plus\dimen@
  \penalty-100\vskip\z@\@plus-\dimen@
  \vskip\dimen@\penalty9999\vskip-\dimen@\vskip\z@skip
  \endgroup
  \smallskip\noindent\emph{#1}}
\makeatother

\title[endpoint eigenfunction estimates]
{Sharp endpoint eigenfunction estimates for the two-dimensional Hermite operator}
\author{Guiyu Xie and Cheng Zhang}
\address{Yau Mathematical Sciences Center, Tsinghua University, Beijing, 100084, P.R. China}
\email{xiegy25@mails.tsinghua.edu.cn; czhang98@tsinghua.edu.cn}
\date{}

\begin{document}
\begin{abstract}
Let \(\calH=-\Delta+|x|^2\) be the Hermite operator on \(\R^2\), and let
\(\Pi_\lambda\) denote the spectral projection corresponding to
\(\lambda=2N+2\). We prove the sharp log-free endpoint estimate
\[
\norm{\Pi_\lambda}_{L^2(\R^2)\to L^{10/3}(\R^2)}
\lesssim \lambda^{-1/10}.
\]
The proof uses a spectral decomposition in polar coordinates and combines Koch--Tataru localized spectral projection bounds with a
Liouville--Green representation, van der Corput estimates for exponential sums,
and a weighted \(TT^*\) argument across radial scales.
\end{abstract}
\maketitle

\section{Introduction}

\subsection{Background}
Consider the Hermite operator
\[
 \mathcal H=-\Delta+|x|^2
\]
on \(\R^d\). For \(n\in\mathbb N_0\), let \(h_n\) denote the standard
\(L^2(\R)\)-normalized one-dimensional Hermite function, satisfying
\((-\partial_t^2+t^2)h_n=(2n+1)h_n\). For a multi-index
\(\alpha=(\alpha_1,\ldots,\alpha_d)\in\mathbb N_0^d\), set
\[
 |\alpha|=\alpha_1+\cdots+\alpha_d,
 \qquad
 h_\alpha(x)=h_{\alpha_1}(x_1)\cdots h_{\alpha_d}(x_d).
\]
The functions \(h_\alpha\) form the Cartesian Hermite orthonormal basis of
\(L^2(\R^d)\) and satisfy
\(\mathcal Hh_\alpha=(2|\alpha|+d)h_\alpha\). For \(\lambda=2N+d\), with
\(N\in\mathbb N_0\), let
\[
 \Pi_\lambda f
 =\sum_{|\alpha|=N}\langle f,h_\alpha\rangle h_\alpha
\]
be the orthogonal projection onto the \(\lambda\)-eigenspace. The problem is
to determine the optimal growth or decay of
\(\|\Pi_\lambda\|_{L^2(\R^d)\to L^p(\R^d)}\). Early estimates of Karadzhov \cite{Karadzhov} and
Thangavelu \cite{Thangavelu,Thangavelu1998}, followed by the global and spatially localized estimates of
Koch--Tataru \cite{KochTataru}, established the sharp powers away from the critical exponent
\[
 p_e(d)=\frac{2(d+3)}{d+1}.
\] Moreover, Wang--Zhang \cite{WangZhang} proved sharp local \(L^p\) bounds on arbitrarily translated and
dilated compact sets, refining the local consequences of the Koch--Tataru
estimates. The problem of obtaining $L^p$ eigenfunction bounds has received considerable interest in the context of Bochner-Riesz means \cite{Thangavelu, thang87, Thangavelu1998,Karadzhov, leeadv,chentams, chenadv,chencr,chenjga}, as well as unique continuation problems \cite{esca,ev, kt09}.

At \(p=p_e(d)\), the annular estimates of Koch--Tataru give the factor
\(\lambda^{-1/(2(d+3))}(\log\lambda)^{1/p_e(d)}\). In dimension one the
analogous logarithmic loss is known to be unavoidable \cite{Markett}. This left
open whether the logarithmic factor was necessary in dimensions at least two.
For \(d\ge3\), Jeong--Lee--Ryu \cite{JLR} removed the logarithmic loss by
asymmetric localization on the two sides of the turning hypersurface. While
this paper was being completed, Jeong--Lee--Ryu posted the preprint
\cite{JLR2D}, which proves the same two-dimensional endpoint estimate through
asymmetric localization, a recursive multiscale decomposition in space and
time, and almost orthogonality.

The proof presented here was developed independently. 
After a spectral decomposition in polar coordinates, a uniform Liouville--Green representation
reduces the non-glancing contribution to an exponential sum, and one
weighted \(TT^*\) argument couples all relevant radial scales. The remaining
modes are handled by standard dyadic decompositions and rapid decay in the
forbidden region. Thus the logarithmic loss is removed without recursive
space--time localization or an iterative almost-orthogonality argument.

\subsection{Main result}
Our main result is the following logarithm-free estimate.

\begin{theorem}\label{thm:main}
There exists an absolute constant \(C\) such that, for every
\(\lambda\in2\N+2\) and every \(f\in L^2(\R^2)\),
\begin{equation}\label{eq:main}
 \norm{\Pi_\lambda f}_{L^{10/3}(\R^2)}
 \le C\lambda^{-1/10}\norm f_{L^2(\R^2)}.
\end{equation}
\end{theorem}

The estimate \eqref{eq:main} is sharp. Indeed,  we consider the  highest-angular-momentum eigenfunction
\[
G_N(x_1,x_2)
=\frac{(x_1+ix_2)^N}{\sqrt{\pi N!}}e^{-|x|^2/2}.
\]
It belongs to \(\operatorname{Ran}\Pi_{2N+2}\) and has \(L^2\)-norm one and 
$
\norm{G_N}_{10/3}\approx N^{-1/10}.
$

Let \(p=10/3\) and \(u=1-|x|^2/\lambda\). On each  annulus
\(\{\mu/2\le u\le2\mu\}\), where
\(\lambda^{-2/3}\lesssim\mu\lesssim1\), 
Koch--Tataru \cite[Theorem~3(a)]{KochTataru} gives
\[
 \norm{\one_{\{\mu/2\le u\le2\mu\}}\Pi_\lambda}_{2\to p}
 \lesssim\lambda^{-1/10}.
\]
Taking the \(\ell^{p}\) sum over the output annuli gives only
\[
 \lambda^{-1/10}
 \bigl(\#\{\text{dyadic annuli}\}\bigr)^{1/p}
 \approx
 \lambda^{-1/10}(\log\lambda)^{3/10}.
\]
The proof of Theorem \ref{thm:main} couples the radial scales and removes this logarithmic
factor.

\subsection{Outline of the proof}
Set \(g=\Pi_\lambda f\) and \(u=1-|x|^2/\lambda\).  The localized estimates of Koch--Tataru in
Section~\ref{sec:local} control the turning-point core, the deep interior,
and the exterior.  It remains to estimate \(g\) on the interior collar
\(\mathcal A_\lambda=\{\lambda^{-2/3}\ll u\ll1\}\). 

On \(\mathcal A_\lambda\), use the polar basis from
Lemma~\ref{lem:polar-basis} and the expansion \eqref{eq:polar-expansion}.
For the \(m\)-th mode, the parameter \(\nu_m\)
in \eqref{eq:nu-definition} gives the outer radial turning point.  The
factorization \eqref{eq:Q-factor-u} shows that \(u>\nu_m\) is the
allowed region, while \(u<\nu_m\) is the
forbidden region.  A mode is called
non-glancing at radius \(u\) when \( u\gg \nu_m\), so it stays a
fixed relative distance from the turning point.

For the non-glancing modes, Lemma~\ref{lem:lg-radial} gives a uniform
Liouville--Green representation.  Lemmas~\ref{lem:phase-second-variation}
and \ref{lem:discrete-vdc}, together with the weighted \(TT^*\) estimate in
Proposition~\ref{prop:weighted-sum}, couple all radial scales and yield
Proposition~\ref{prop:below-threshold}.

Comparable scales $u\approx \nu_m$ are controlled by the localized
Koch--Tataru estimates and the uniform angular multiplier bounds.  When
\(u\ll \nu_m\), Lemma~\ref{lem:forbidden-radial} gives rapid decay in the
forbidden region.  Then orthogonality and  Young's inequality yield
Proposition~\ref{prop:complementary}.  Combining these estimates proves
Theorem~\ref{thm:main}.  

\subsection{Organization of the paper}
 Section~\ref{sec:local} records the localized
Koch--Tataru estimates in the parameters used here.
Section~\ref{sec:polar} introduces the polar decomposition and the
mode-dependent turning parameter. Section~\ref{sec:radial} proves the
Liouville--Green representation in the allowed region and rapid decay in the
forbidden region. Section~\ref{sec:weighted} establishes the weighted 
exponential-sum estimate. Section~\ref{sec:non-glancing} applies this estimate
to the non-glancing modes. Section~\ref{sec:complementary} treats the remaining
modes by dyadic decompositions and orthogonality. Section~\ref{sec:proof-main}
combines the estimates to prove Theorem \ref{thm:main}.

\subsection{Notation}
Throughout the paper, \(X\lesssim Y\) means that \(X\le CY\), and
\(X\approx Y\) means that both \(X\lesssim Y\) and \(Y\lesssim X\) hold.
Implicit constants are independent of the eigenvalue. Moreover,
\(X\gg Y\) means that \(X\ge CY\) for a sufficiently large fixed constant
\(C\).

\subsection{Acknowledgments}
The authors are supported in part by the National Key R\&D Program of China
2024YFA1015300.  C.Z. is also supported in part by NSFC Grant 12371097.
\subsection{Statements and Declarations}
Data sharing not applicable to this article as no datasets were generated or analyzed during the current study. The authors have no relevant financial or non-financial interests to disclose.
\section{Koch--Tataru estimates}\label{sec:local}

Koch--Tataru use \(\rho^2\) for the eigenvalue \(\lambda=2N+2\). Recall that
\(u(x)=1-|x|^2/\lambda\), and set
\[
 \begin{gathered}
  \rho=\sqrt\lambda,\qquad
  \eta(x)=\frac{|x|}{\sqrt\lambda}-1,
  \qquad u_*=\lambda^{-2/3},\\
  y=\rho^{-2/3}(\rho^2-|x|^2)=\lambda^{2/3}u(x).
 \end{gathered}
\]
Thus \(u\) is the signed normalized radial coordinate relative to the turning
circle \(\{|x|=\sqrt\lambda\}\) and is positive on its interior side.
The quantity \(u_*\) is the turning-point scale. For \(\mu>0\), write
\[
 \mathcal A_\mu^{\rm in}=\{x:\tfrac12\mu\le u(x)\le2\mu\},
 \qquad
 \mathcal A_\mu^{\rm out}=\{x:\tfrac12\mu\le\eta(x)\le2\mu\}.
\]
Since
\[
 u=(1-|x|/\rho)(1+|x|/\rho),
 \qquad -y=\lambda^{2/3}(2\eta+\eta^2),
\]
each of these annuli meets only a bounded number of pieces in the
Koch--Tataru turning-point decomposition.
Choose fixed constants
\begin{equation}\label{eq:lg-constant-hierarchy}
 \begin{gathered}
  0<u_0<U_1<U_2<\frac14,\qquad 0<\eta_0<\frac14,\\
  0<\kappa_0<\kappa_1<\frac14,\qquad 0<c_1<1,
 \end{gathered}
\end{equation}
with \(u_0\) and \(\eta_0\) sufficiently small that the bounded-overlap
constants above are uniform for \(0<\mu\le u_0\) and
\(0<\mu\le\eta_0\), respectively. Choose
\(C_*\ge\max\{1,2/c_1\}\) sufficiently large, depending only on these fixed
constants, for the estimates in Sections~\ref{sec:radial}--\ref{sec:complementary}.
All implicit constants below are independent of \(N,\lambda,m,j,k\). We may
assume \(\lambda\ge\lambda_*\) for a fixed \(\lambda_*\), since the projection
norms associated with the finitely many smaller eigenvalues can be absorbed
into the implicit constants.

We first state the precise input. For
\(g\in\operatorname{Ran}\Pi_\lambda\), one has
\((\mathcal H-\rho^2)g=0\). We retain the notation
\(\ell^\infty_\rho L^p\) of Koch--Tataru for the supremum of the
\(L^p\)-norms over the pieces of their turning-point decomposition. In
dimension two, \cite[Theorem~3(a)]{KochTataru} gives, for \(2\le p\le6\),
\begin{equation}\label{eq:KT-weighted}
 \left\|
 \rho^{\,1/3-(2/3)\delta_p}
 \langle y\rangle_+^{-1/4+(5/4)\delta_p}
 \langle y\rangle_-^{1-\delta_p}g
 \right\|_{\ell^\infty_\rho L^p}
 \lesssim\norm g_2,
 \qquad
 \delta_p=\frac12-\frac1p,
\end{equation}
where \(\langle y\rangle_\pm=1+\max(\pm y,0)\).

\begin{proposition}
If \(C_*u_*\le\mu\le u_0\), then
\begin{equation}\label{eq:package-allowed}
 \norm{\one_{\mathcal A_\mu^{\rm in}}\Pi_\lambda}_{2\to10/3}
 \lesssim\lambda^{-1/10}.
\end{equation}
For \(u_*\le\mu\le u_0\),
\begin{equation}\label{eq:package-inner-L2}
 \norm{\one_{\mathcal A_\mu^{\rm in}}\Pi_\lambda}_{2\to2}
 \lesssim\mu^{1/4}.
\end{equation}
For every fixed \(A\ge1\),
\begin{equation}\label{eq:package-core}
 \norm{\one_{\{|u|\le Au_*\}}\Pi_\lambda}_{2\to10/3}
 \lesssim\lambda^{-1/10},
\end{equation}
and the fixed deep-interior region satisfies
\begin{equation}\label{eq:package-deep}
 \norm{\one_{\{u\ge u_0\}}\Pi_\lambda}_{2\to10/3}
 \lesssim\lambda^{-1/10}.
\end{equation}
If \(C_*u_*\le\mu\le\eta_0\), then
\begin{equation}\label{eq:package-outer}
 \norm{\one_{\mathcal A_\mu^{\rm out}}\Pi_\lambda}_{2\to10/3}
 \lesssim(\lambda\mu)^{-3/10},
\end{equation}
while
\begin{equation}\label{eq:package-far-outer}
 \norm{\one_{\{\eta\ge\eta_0/2\}}\Pi_\lambda}_{2\to10/3}
 \lesssim\lambda^{-1/10}.
\end{equation}
\end{proposition}

\begin{proof}
Let \(g=\Pi_\lambda f\), so \(\norm g_2\le\norm f_2\). Restricting
\eqref{eq:KT-weighted} to the three types of pieces gives
\begin{align*}
 \norm{\one_{\mathcal A_\mu^{\rm in}}g}_p
 &\lesssim
 \lambda^{-\delta_p/2}
 \mu^{1/4-(5/4)\delta_p}\norm g_2,
 &&u_*\lesssim\mu\le u_0,\\
 \norm{\one_{\{|u|\le Au_*\}}g}_p
 &\lesssim
 \lambda^{-1/6+\delta_p/3}\norm g_2,\\
 \norm{\one_{\mathcal A_\mu^{\rm out}}g}_p
 &\lesssim
 \lambda^{-5/6+\delta_p}
 \mu^{-(1-\delta_p)}\norm g_2,
 &&C_*u_*\le\mu\le\eta_0.
\end{align*}
Here the middle estimate follows because the indicated set meets only
\(O(1)\) turning-point pieces.

For \(p=10/3\), one has \(\delta_p=1/5\), and hence
\[
 -\frac{\delta_p}{2}=-\frac1{10},\qquad
 \frac14-\frac54\delta_p=0,
 \qquad
 -\frac16+\frac13\delta_p=-\frac1{10}.
\]
The first two estimates therefore prove \eqref{eq:package-allowed} and
\eqref{eq:package-core}. The first estimate with \(p=2\) proves
\eqref{eq:package-inner-L2} when \(\mu\ge C_*u_*\).
For the remaining range
\(u_*\le\mu<C_*u_*\), the middle estimate
with \(p=2\) and \(A=2C_*\) gives
\[
 \lambda^{-1/6}=u_*^{1/4}\le\mu^{1/4}.
\]
At
\(p=10/3\), the positive-\(y\) weight in \eqref{eq:KT-weighted} also has
exponent zero. Since
\(\{u\ge u_0\}\) meets only \(O(1)\) pieces, this proves
\eqref{eq:package-deep}.

The exterior estimate at \(p=10/3\) is
\[
 \norm{\one_{\mathcal A_\mu^{\rm out}}\Pi_\lambda}_{2\to10/3}
 \lesssim\lambda^{-19/30}\mu^{-4/5}
 =(\lambda\mu)^{-3/10}(\lambda^{2/3}\mu)^{-1/2}.
\]
Because \(\mu\ge C_*u_*\), it implies
\eqref{eq:package-outer}.

For the far exterior, \cite[Theorem~3(b)]{KochTataru} with \(p=\infty\)
and decay exponent one gives
\[
 \norm{\langle y\rangle_-\Pi_\lambda f}_{L^\infty(D_\rho^{\rm ext})}
 \lesssim\norm f_2,
 \qquad
 D_\rho^{\rm ext}=\{|x|>\rho+\tfrac12\rho^{-1/3}\}.
\]
For large \(\lambda\), the set \(\{\eta\ge\eta_0/2\}\) lies in
\(D_\rho^{\rm ext}\), and \(\langle y\rangle_-\gtrsim\lambda^{2/3}\)
there. Thus the corresponding \(L^2\to L^\infty\) norm is
\(O(\lambda^{-2/3})\). Interpolation with the \(L^2\) bound gives
\(\lambda^{-4/15}\le\lambda^{-1/10}\), proving
\eqref{eq:package-far-outer}. The finitely many remaining eigenvalues are
absorbed into the constant.
\end{proof}

The estimate \eqref{eq:package-outer} is summable over the exterior annuli.
Indeed, for \(g=\Pi_\lambda f\) and dyadic
\(\mu\in[C_*u_*,\eta_0]\),
\begin{align}
 \norm{\one_{\{C_*u_*\le\eta\le\eta_0\}}g}_{10/3}^{10/3}
 &\lesssim\sum_\mu
 \norm{\one_{\mathcal A_\mu^{\rm out}}g}_{10/3}^{10/3}\notag\\
 &\lesssim\lambda^{-1}\sum_\mu\mu^{-1}\norm f_2^{10/3}
 \lesssim\lambda^{-1/3}\norm f_2^{10/3}.
 \label{eq:outer-sum}
\end{align}
Here
\(\sum_\mu\mu^{-1}\lesssim(C_*u_*)^{-1}\approx\lambda^{2/3}\).
It remains to estimate the interior collar
\[
 \mathcal A_\lambda=\{x:C_*u_*\le u(x)\le u_0\}
\]
by the angular decomposition.

\section{Spectral decomposition in polar coordinates}\label{sec:polar}

Write \(\T=\R/(2\pi\mathbb Z)\). In polar coordinates
\(x=(r\cos\theta,r\sin\theta)\), with \(\theta\in\T\), set
\[
 \mathcal M_N=\{m\in\mathbb Z:|m|\le N,\ m\equiv N\pmod2\},
 \qquad
 n_m=\frac{N-|m|}{2}\in\mathbb N_0.
\]
Here \(L_n^\alpha\) denotes the generalized Laguerre polynomial of degree \(n\)
and parameter \(\alpha\), normalized as in
\cite[(1.1.37)]{Thangavelu}. For \(m\in\mathcal M_N\), let
\begin{equation}\label{eq:radial-laguerre}
 \begin{split}
  R_{N,m}(r)&=
  \left(\frac{2n_m!}{(n_m+|m|)!}\right)^{1/2}
  r^{|m|}L_{n_m}^{|m|}(r^2)e^{-r^2/2},\\
  \Phi_{N,m}(r,\theta)&=\frac1{\sqrt{2\pi}}R_{N,m}(r)e^{im\theta}.
 \end{split}
\end{equation}

\begin{lemma}[Polar Laguerre--Hermite basis]\label{lem:polar-basis}
For every \(m\in\mathcal M_N\),
\[
 \int_0^\infty|R_{N,m}(r)|^2r\dd r=1.
\]
The family
\(\{\Phi_{N,m}:m\in\mathcal M_N\}\) is an orthonormal basis of
\(\operatorname{Ran}\Pi_\lambda\), where
\(\lambda=2N+2\). Consequently, every
\(g\in\operatorname{Ran}\Pi_\lambda\) has the unique expansion
\begin{equation}\label{eq:polar-expansion}
 \begin{gathered}
  g(r,\theta)=\frac1{\sqrt{2\pi}}
  \sum_{m\in\mathcal M_N}c_mR_{N,m}(r)e^{im\theta},\\
  \sum_{m\in\mathcal M_N}|c_m|^2=\norm g_2^2.
 \end{gathered}
\end{equation}
\end{lemma}

\begin{proof}
The form of the basis is the two-dimensional case of the Hecke--Bochner
formula for Hermite projections \cite[Theorem~3.4.1]{Thangavelu}. For
completeness, we verify it directly.

Fix \(m\in\mathcal M_N\), and put \(q=|m|\) and \(k=n_m\), so that
\(2k+q=N\). In polar coordinates,
\[
 \mathcal H=-\partial_r^2-r^{-1}\partial_r-r^{-2}\partial_\theta^2+r^2.
\]
The function \(\ell(s)=L_k^q(s)\) satisfies the Laguerre equation
\[
 s\ell''(s)+(q+1-s)\ell'(s)+k\ell(s)=0.
\]
Consequently, for \(F(r)=r^qe^{-r^2/2}\ell(r^2)\), direct differentiation
gives
\[
 \left(-\partial_r^2-r^{-1}\partial_r+q^2r^{-2}+r^2\right)F
 =(4k+2q+2)F.
\]
Since \(m^2=q^2\) and \(2k+q=N\),
\[
 \mathcal H\bigl(F(r)e^{im\theta}\bigr)
 =(2N+2)F(r)e^{im\theta}.
\]
Moreover,
\[
 r^qe^{im\theta}=
 \begin{cases}
  (x_1+ix_2)^q,&m\ge0,\\
  (x_1-ix_2)^q,&m<0,
 \end{cases}
\]
so \(F(r)e^{im\theta}\) is a polynomial times \(e^{-|x|^2/2}\). It is
therefore smooth at the origin and belongs to \(\mathcal S(\R^2)\). Thus
\(\Phi_{N,m}\in\operatorname{Ran}\Pi_\lambda\).

The Laguerre orthogonality relation
\cite[(5.1.1)]{Szego} gives
\[
 \int_0^\infty s^q|L_k^q(s)|^2e^{-s}\dd s=\frac{(k+q)!}{k!}.
\]
Hence the change of variables \(s=r^2\), together with
\eqref{eq:radial-laguerre}, yields
\begin{align*}
 \int_0^\infty|R_{N,m}(r)|^2r\dd r
 &=\frac{2k!}{(k+q)!}
   \int_0^\infty r^{2q}|L_k^q(r^2)|^2e^{-r^2}r\dd r\\
 &=\frac{k!}{(k+q)!}
   \int_0^\infty s^q|L_k^q(s)|^2e^{-s}\dd s
 =1.
\end{align*}
The functions \((2\pi)^{-1/2}e^{im\theta}\) are orthonormal on \(\T\).
Therefore the functions \(\Phi_{N,m}\), \(m\in\mathcal M_N\), form an
orthonormal family in \(\operatorname{Ran}\Pi_\lambda\).

Finally, \(\mathcal M_N=\{-N,-N+2,\ldots,N\}\), so
\(\#\mathcal M_N=N+1\). By the Cartesian Hermite decomposition in the
Introduction,
\[
 \dim\operatorname{Ran}\Pi_\lambda
 =\#\{\alpha\in\mathbb N_0^2:|\alpha|=N\}=N+1.
\]
The orthonormal family therefore has the full dimension of the eigenspace and
is a basis. The unique expansion \eqref{eq:polar-expansion} and the identity
for its coefficients now follow from Parseval's identity.
\end{proof}

Set \(v_m(r)=r^{1/2}R_{N,m}(r)\). The radial Hermite equation gives
\[
 v_m''(r)+Q_m(r)v_m(r)=0,
 \qquad
 Q_m(r)=\lambda-r^2-\frac{m^2-1/4}{r^2}.
\]
For each \(m\in\mathcal M_N\), set
\begin{equation}\label{eq:nu-definition}
 \nu_m=\frac{1-\sqrt{1-4(m^2-1/4)/\lambda^2}}2,
 \qquad
 \nu_m(1-\nu_m)=\frac{m^2-1/4}{\lambda^2}.
\end{equation}
The square root in \eqref{eq:nu-definition} is real because
\(4(m^2-\tfrac14)\le4N^2-1<\lambda^2\). Moreover, \(\nu_0<0\) when
\(0\in\mathcal M_N\), whereas \(0<\nu_m<1/2\) for \(|m|\ge1\). For
\(|m|\ge1\),
\begin{equation}\label{eq:nu-m-comparable}
 \nu_m\approx\frac{m^2}{\lambda^2}.
\end{equation}
Indeed, \(1/2<1-\nu_m<1\) and
\(3m^2/4\le m^2-1/4\le m^2\).
Put \(r(u)=\sqrt{\lambda(1-u)}\). Then the exact factorization
\begin{equation}\label{eq:Q-factor-u}
 Q_m(r(u))=
 \lambda\frac{(u-\nu_m)(1-u-\nu_m)}{1-u}
\end{equation}
holds. In the collar \(u<1/4\), one has \(1-u-\nu_m>1/4\). Thus
\(Q_m(r(u))\) has the sign of \(u-\nu_m\): \(u>\nu_m\) is the
allowed region, whereas \(u<\nu_m\) is the forbidden region. For \(|m|\ge1\), \(\nu_m\)
is exactly the \(u\)-coordinate of the outer radial turning point. When
present, the value \(\nu_0<0\) is retained for uniform notation.

\section{Radial mode estimates}\label{sec:radial}

This section establishes the radial estimates needed for the angular-mode
decomposition. In the allowed region, away from the turning
point, we derive a uniform Liouville--Green representation and prove the
curvature bound for differences of the resulting phases. In the forbidden
region, we obtain rapid decay of the radial modes. 

\subsection{Liouville--Green representation away from the turning point}

\noindent\hspace*{\parindent}The Liouville--Green representation provides the
oscillatory form needed to apply Proposition~\ref{prop:weighted-sum} to the
non-glancing modes. Away from a turning point, the construction is classical;
see \cite[Chapter~6, Sections~1--5]{Olver} and \cite{OlverLG}. For completeness,
we give a detailed proof.

For \(u\in[C_*u_*,U_1]\), set
\[
 \mathcal I(u)=\{m\in\mathcal M_N:|m|\le\kappa_1\lambda\sqrt u\}.
\]
For real \(\xi\) with \(|\xi|\le\kappa_1\lambda\sqrt u\), define
\[
 Q_\xi(r)=\lambda-r^2-\frac{\xi^2-1/4}{r^2},
 \qquad
 S(u,\xi)=\int_{r(U_1)}^{r(u)}Q_\xi(s)^{1/2}\dd s.
\]
The phase is well defined. Indeed, for \(u\le v\le U_1\),
\[
 Q_\xi(r(v))
 =\lambda v-\frac{\xi^2-1/4}{\lambda(1-v)}
 \approx\lambda v>0.
\]
For \(m\in\mathcal I(u)\), put
\[
 b(u,m)=\left[\left(1-\frac{\nu_m}{u}\right)
                (1-u-\nu_m)\right]^{-1/4}.
\]

\begin{lemma}\label{lem:lg-radial}
For every \(m\in\mathcal I(U_1)\), there exist coefficients
\(d_{m,+},d_{m,-}\), depending only on \((\lambda,m)\), such that
\(|d_{m,\pm}|\le C\) and, whenever
\(u\in[C_*u_*,U_1]\) and \(m\in\mathcal I(u)\),
\begin{equation}\label{eq:lg-expansion}
 R_{N,m}(r(u))
 =\lambda^{-1/2}u^{-1/4}b(u,m)
  \sum_{\sigma=\pm}d_{m,\sigma}e^{i\sigma S(u,m)}
  +\mathcal E_m(u).
\end{equation}
Moreover,
\begin{equation}\label{eq:lg-symbol}
 |b(u,m)|\le C,
 \qquad
 |b(u,m+2)-b(u,m)|\le C(\lambda\sqrt u)^{-1},
\end{equation}
where the second estimate holds when
\(m,m+2\in\mathcal I(u)\), and
\begin{equation}\label{eq:lg-error}
 |\mathcal E_m(u)|
 \le C\lambda^{-1/2}u^{-1/4}(\lambda u^{3/2})^{-1}.
\end{equation}
\end{lemma}

\begin{proof}
Fix \(m\in\mathcal I(U_1)\), write \(p=Q_m^{1/2}\), and set
\[
 z=z(r)=\int_{r(U_1)}^r p(s)\dd s,
 \qquad
 w(z)=p(r)^{1/2}v_m(r).
\]
Direct differentiation gives
\[
 w_{zz}+w=V_mw,
 \qquad
 V_m=\frac{p''}{2p^3}-\frac{3(p')^2}{4p^4}.
\]

Let \(J=[r(U_2),r(U_1)]\). Since \(m\in\mathcal I(U_1)\), the displayed formula for
\(Q_\xi(r(v))\), with the same direct calculation for
\(U_1\le v\le U_2\), gives
\[
 |J|\approx\sqrt\lambda,
 \qquad p(r)^2\approx\lambda\quad(r\in J).
\]
Since \(v_m=r^{1/2}R_{N,m}\),
\[
 |v_m'|^2\le2r|R_{N,m}'|^2+(2r)^{-1}|R_{N,m}|^2.
\]

 Write \(R=R_{N,m}\).
The radial eigenvalue equation is
\[
 -\frac1r(rR')'+\frac{m^2}{r^2}R+r^2R=\lambda R.
\]
Multiplying by \(\overline R\,r\) and integrating by parts gives the radial energy identity
\[
 \int_0^\infty\left(|R'|^2+\frac{m^2}{r^2}|R|^2
 +r^2|R|^2\right)r\dd r
 =\lambda\int_0^\infty|R|^2r\dd r=\lambda.
\]
Here the boundary term \([rR'\overline R]_0^\infty\) vanishes. Indeed,
\eqref{eq:radial-laguerre} gives \(R=O(r^{|m|})\) and
\(R'=O(r^{|m|-1})\) near the origin when \(|m|\ge1\), while
\(R'=O(r)\) when \(m=0\); at infinity both \(R\) and \(R'\) decay
rapidly because of the Gaussian factor. The last equality uses
Lemma~\ref{lem:polar-basis}.
Together with \(\norm{v_m}_2=1\) and \(r\approx\sqrt\lambda\) on \(J\),
this yields
\[
 \int_J\bigl(|v_m'|^2+\lambda|v_m|^2\bigr)\dd r\lesssim\lambda.
\]
Hence there is \(\tilde r_m\in J\) such that
\[
 |v_m'(\tilde r_m)|\lesssim\lambda^{1/4},
 \qquad |v_m(\tilde r_m)|\lesssim\lambda^{-1/4}.
\]
Put \(\tilde z_m=z(\tilde r_m)\). On \(J\),
\(p\approx\lambda^{1/2}\) and \(|p'|\lesssim1\); therefore
\[
 |w(\tilde z_m)|+|w_z(\tilde z_m)|\lesssim1.
\]
Define
\[
 d_{m,\pm}=\frac12e^{\mp i\tilde z_m}
 \bigl(w(\tilde z_m)\mp iw_z(\tilde z_m)\bigr).
\]
Then \(|d_{m,\pm}|\lesssim1\), and Duhamel's formula is
\begin{equation}\label{eq:lg-duhamel}
 w(z)=\sum_{\sigma=\pm}d_{m,\sigma}e^{i\sigma z}
 +\int_{\tilde z_m}^z\sin(z-\zeta)V_m(\zeta)w(\zeta)\dd\zeta.
\end{equation}

Now let \(u\in[C_*u_*,U_1]\) with \(m\in\mathcal I(u)\).
Along the interval from \(\tilde r_m\) to \(r(u)\), write \(r=r(v)\).
Then \(u\le v\le U_2\), and direct differentiation gives
\[
 p(r(v))\approx\lambda^{1/2}v^{1/2},
 \qquad |p'(r(v))|\lesssim v^{-1/2},
 \qquad |p''(r(v))|\lesssim\lambda^{-1/2}v^{-3/2}.
\]
Consequently,
\[
 |V_m|\lesssim\lambda^{-2}v^{-3},
 \qquad |\dd z|=p|\dd r|\approx\lambda v^{1/2}\dd v,
\]
and hence
\[
 \int_{\tilde z_m}^{S(u,m)}|V_m(\zeta)|\dd\zeta
 \lesssim\lambda^{-1}\int_u^{U_2}v^{-5/2}\dd v
 \lesssim(\lambda u^{3/2})^{-1}.
\]
Gronwall's inequality applied to \eqref{eq:lg-duhamel} gives
\(\sup_{\tilde z_m\le z\le S(u,m)}|w(z)|\lesssim1\). Thus the integral
remainder in \eqref{eq:lg-duhamel}, evaluated at \(z=S(u,m)\), is
\(O((\lambda u^{3/2})^{-1})\).

Since \(R_{N,m}=r^{-1/2}p^{-1/2}w\), the factorization
\eqref{eq:Q-factor-u} gives
\[
 r(u)^{-1/2}p(r(u))^{-1/2}
 =\lambda^{-1/2}
  \bigl[(u-\nu_m)(1-u-\nu_m)\bigr]^{-1/4}
 =\lambda^{-1/2}u^{-1/4}b(u,m).
\]
This proves \eqref{eq:lg-expansion} and \eqref{eq:lg-error}. The reference
point \(\tilde r_m\) depends only on \((\lambda,m)\), so the same
coefficients work for every admissible \(u\).

It remains to prove \eqref{eq:lg-symbol}. If
\(k\in\mathcal I(u)\) and \(|k|\ge1\), then
\[
 0<\nu_k\le 2\frac{k^2}{\lambda^2}
 \le2\kappa_1^2u\le\frac u8.
\]
For \(k=0\), one has \(|\nu_0|\le(4\lambda^2)^{-1}\le u/8\), after
increasing the fixed lower threshold \(\lambda_*\) if necessary. Hence
\[
 1-\frac{\nu_k}{u}\approx1,
 \qquad
 1-u-\nu_k\approx1,
\]
which proves \(|b(u,k)|\lesssim1\). If
\(m,m+2\in\mathcal I(u)\), then
\[
 (\nu_{m+2}-\nu_m)(1-\nu_{m+2}-\nu_m)
 =\frac{4(m+1)}{\lambda^2}.
\]
Since \(1-\nu_{m+2}-\nu_m\approx1\) and
\(|m+1|\lesssim\lambda\sqrt u\),
\[
 |\nu_{m+2}-\nu_m|\lesssim\frac{\sqrt u}{\lambda}.
\]
On the interval joining \(\nu_m\) and \(\nu_{m+2}\), the derivative with
respect to \(\nu\) of
\(\bigl[(1-\nu/u)(1-u-\nu)\bigr]^{-1/4}\) is \(O(u^{-1})\).
The mean value theorem therefore yields
\[
 |b(u,m+2)-b(u,m)|\lesssim(\lambda\sqrt u)^{-1}.
\]
\end{proof}

\begin{lemma}[Curvature of the radial phase]\label{lem:phase-second-variation}
For \(\tau>0\), write \(u_\tau=\lambda^2\tau^2\). Let \(s,t>0\) satisfy
\[
 u_s,u_t\in[C_*u_*,U_1].
\]
Uniformly for every real \(\xi\) satisfying
\(|\xi|\le\kappa_1\lambda^2\min(s,t)\),
\[
 \left|\partial_\xi^2\bigl(S(u_t,\xi)-S(u_s,\xi)\bigr)\right|
 \approx|t-s|.
\]
If \(t\ne s\), this derivative has the sign of \(t-s\) throughout the
stated interval.
\end{lemma}

\begin{proof}
Let \(p_\xi=Q_\xi^{1/2}\). For \(v\) between \(u_s\) and \(u_t\),
the explicit formula for \(Q_\xi(r(v))\) and the bound on \(\xi\) give
\[
 p_\xi(r(v))^2\approx\lambda v,
 \qquad r(v)^2\approx\lambda,
 \qquad |\dd r|\approx\lambda^{1/2}\dd v.
\]
The common lower limit \(r(U_1)\) cancels, and the endpoints do not depend
on \(\xi\). Hence
\[
 S(u_t,\xi)-S(u_s,\xi)
 =\int_{r(u_s)}^{r(u_t)}p_\xi(r)\dd r,
 \qquad
 \partial_\xi^2p_\xi
 =-\frac1{r^2p_\xi}-\frac{\xi^2}{r^4p_\xi^3}.
\]
Since \(\xi^2/(r^2p_\xi^2)\lesssim\kappa_1^2\),
\[
 -\partial_\xi^2p_\xi\approx\frac1{r^2p_\xi}.
\]
Therefore
\[
 \left|\partial_\xi^2\bigl(S(u_t,\xi)-S(u_s,\xi)\bigr)\right|
 \approx\frac1\lambda
 \left|\int_{u_s}^{u_t}v^{-1/2}\dd v\right|
 =2|t-s|.
\]
Because \(\partial_\xi^2p_\xi<0\) and \(r(u)\) is decreasing, the derivative
has the sign of \(t-s\).
\end{proof}

\subsection{Rapid decay in the forbidden region}

If \(u\le c_1\nu_m\), then
\(\nu_m-u\ge(1-c_1)\nu_m\), so the coefficient in the rescaled equation
below is uniformly positive.

\begin{lemma}\label{lem:forbidden-radial}
Suppose that
\[
 0\le u\le c_1\nu_m,
 \qquad \nu_m\ge C_*u_*.
\]
Then, for every \(M>0\),
\begin{equation}\label{eq:forbidden-radial}
 |R_{N,m}(r(u))|
 \lesssim_M
 \lambda^{-1/2}\nu_m^{-1/4}
 \bigl(\lambda\nu_m^{3/2}\bigr)^{-M}.
\end{equation}
\end{lemma}

\begin{proof}
Put
\[
 \nu=\nu_m,\qquad \Lambda=\lambda\nu^{3/2},\qquad y=\frac u\nu,
 \qquad H(y)=(1-\nu y)^{1/4}v_m(r(\nu y)).
\]
A direct calculation from the radial equation gives
\[
 H''=\Lambda^2F_{\nu,\Lambda}(y)H,
 \qquad
 F_{\nu,\Lambda}(y)
 =\frac{(1-y)(1-\nu-\nu y)}{4(1-\nu y)^2}
 -\frac{3\nu^2}{16\Lambda^2(1-\nu y)^2}.
\]
Set \(\beta=(1+c_1)/2\). For \(y\le\beta\), write \(a=1-y\). The first
term in \(F_{\nu,\Lambda}\) satisfies
\[
 \frac{a(1-2\nu+\nu a)}{4(1-\nu+\nu a)^2}
 \ge \frac{a^2}{2(1+a)^2}\ge c>0.
\]
Indeed, after removing the factor \(a/4\), the quotient on the left is
decreasing for \(0<\nu<1/2\), and its value at \(\nu=1/2\) is
\(2a/(1+a)^2\). The second term is \(O(\Lambda^{-2})\). Since
\(\Lambda\ge C_*^{3/2}\), the choice of \(C_*\) therefore ensures that
\begin{equation}\label{eq:forbidden-positivity}
 F_{\nu,\Lambda}(y)\ge c_0>0\qquad(y\le\beta).
\end{equation}

Fix \(c_1<\gamma<\gamma'<\beta\). The localized estimate
\eqref{eq:package-inner-L2}, applied to the normalized mode
\(\Phi_{N,m}\), gives
\begin{equation}\label{eq:forbidden-mode-strip}
 \left\|\one_{\{\gamma\nu\le u\le\gamma'\nu\}}\Phi_{N,m}\right\|_2
 \lesssim \nu^{1/4}.
\end{equation}
For sufficiently small \(\nu\), this follows by covering the strip by a
bounded number of annuli \(\mathcal A_\mu^{\rm in}\); the lower restriction
\(\mu\ge u_*\) follows from \(\nu\ge C_*u_*\). For the remaining fixed range
of \(\nu\), it follows from \(\|\Phi_{N,m}\|_2=1\). Since
\[
 R_{N,m}(r(\nu y))
 =\lambda^{-1/4}(1-\nu y)^{-1/2}H(y),
 \qquad r\,\dd r=-\frac{\lambda\nu}{2}\,\dd y,
\]
\eqref{eq:forbidden-mode-strip} yields
\begin{equation}\label{eq:forbidden-strip-L2}
 \norm H_{L^2(\gamma,\gamma')}
 \lesssim(\lambda\nu)^{-1/4}.
\end{equation}

By \eqref{eq:radial-laguerre}, \(H(y),H'(y)\to0\) as \(y\to-\infty\).
Let \(W=|H|^2\). On \(( -\infty,\beta]\),
\[
 W''=2|H'|^2+2\Lambda^2F_{\nu,\Lambda}|H|^2
 \ge2c_0\Lambda^2W.
\]
Set \(\kappa=\sqrt{2c_0}\,\Lambda\). Then \(W''\ge\kappa^2W\). Since
\(W'(-\infty)=0\), one has \(W'\ge0\), and hence
\[
 \bigl((W')^2-\kappa^2W^2\bigr)'
 =2W'(W''-\kappa^2W)\ge0.
\]
The expression on the left tends to zero as \(y\to-\infty\). Thus
\(W'\ge\kappa W\), and consequently
\[
 W(y)\le e^{-\kappa(z-y)}W(z),
 \qquad y\le z\le\beta.
\]
Choose \(z\in(\gamma,\gamma')\) with
\(|H(z)|\lesssim\norm H_{L^2(\gamma,\gamma')}\). Then
\eqref{eq:forbidden-strip-L2} gives
\[
 \sup_{y\le c_1}|H(y)|
 \lesssim(\lambda\nu)^{-1/4}e^{-c\Lambda}
 \lesssim_M(\lambda\nu)^{-1/4}\Lambda^{-M}.
\]
Finally,
\[
 R_{N,m}(r(u))
 =\lambda^{-1/4}(1-u)^{-1/2}H(u/\nu).
\]
Since \(u\le c_1\nu<1/2\), this proves \eqref{eq:forbidden-radial}.
\end{proof}

\section{Weighted exponential-sum estimate}\label{sec:weighted}

This section proves the exponential-sum estimate used in the analysis of the
non-glancing modes. We first combine discrete van der Corput with Abel
summation to control exponential sums whose amplitudes have bounded variation.
A weighted $TT^*$ argument then yields Proposition~\ref{prop:weighted-sum},
which couples the radial parameter across all scales and avoids a logarithmic
loss. 

For a finite consecutive
subset \(\Gamma\) of a translate of \(2\mathbb Z\), define
\[
 \operatorname{Var}_\Gamma(a)
 =\sum_{\substack{m\in\Gamma\\m+2\in\Gamma}}|a(m+2)-a(m)|.
\]

\begin{lemma}[Weighted discrete van der Corput estimate]
\label{lem:discrete-vdc}
Let \(\Gamma\) be a finite consecutive subset of a translate of
\(2\mathbb Z\), put \(L=\#\Gamma\), and let \(0<\delta\le1\). Suppose that
\(\psi:\Gamma\to\mathbb R\) extends to a real \(C^2\) function on the convex
hull of \(\Gamma\), that \(\psi''\) has constant sign, and that
\[
 |\psi''(x)|\approx\delta.
\]
Then, for every \(a:\Gamma\to\mathbb C\),
\[
 \left|\sum_{m\in\Gamma}a(m)e^{i\psi(m)}\right|
 \lesssim\bigl(L\delta^{1/2}+\delta^{-1/2}\bigr)
 \bigl(\norm a_{\ell^\infty(\Gamma)}+\operatorname{Var}_\Gamma(a)\bigr).
\]
\end{lemma}

\begin{proof}
For every consecutive \(\Gamma_0\subset\Gamma\), reindex the parity lattice by
consecutive integers and apply the standard discrete van der Corput estimate
\cite[Chapter~2, Theorem~2.2]{GrahamKolesnik}. When
\(\#\Gamma_0\le2\), the same bound is immediate. Hence
\[
 \sup_{\Gamma_0}
 \left|\sum_{m\in\Gamma_0}e^{i\psi(m)}\right|
 \lesssim L\delta^{1/2}+\delta^{-1/2}.
\]
Abel summation on the ordered lattice proves the result.
\end{proof}

\begin{proposition}\label{prop:weighted-sum}
Let \(\lambda\ge1\), let \(I_\lambda\subset(0,\lambda^{-1})\) be an interval,
and let \(\Gamma\) be a finite consecutive subset of a translate of
\(2\mathbb Z\). Let
\(a:I_\lambda\times\Gamma\to\mathbb C\) and
\(\phi:I_\lambda\times\Gamma\to\mathbb R\) be measurable in \(t\), and set
\[
 \mathcal U(t)c(\theta)
 =\sum_{m\in\Gamma}a(t,m)e^{i(m\theta+\phi(t,m))}c_m.
\]
Assume uniformly in \(t\) that \(\operatorname{supp}a(t,\cdot)\) is contained
in a consecutive lattice interval with \(\lesssim1+\lambda^2t\) points and
\[
 \norm{a(t,\cdot)}_{\ell^\infty(\Gamma)}
 +\operatorname{Var}_\Gamma(a(t,\cdot))\lesssim1.
\]
Whenever \(t\ne s\) and the common support contains at least three points,
assume that \(\phi(t,\cdot)-\phi(s,\cdot)\) extends to a real \(C^2\) function
\(\Phi_{t,s}\) on its convex hull and
\[
 |\Phi_{t,s}''(x)|\approx|t-s|.
\]
Then, for every \(c\in\ell^2(\Gamma)\),
\begin{equation}\label{eq:weighted-sum-bound}
 \norm{t^{-1/5}\mathcal U(t)c}_{L^{10/3}(I_\lambda\times\T)}
 \lesssim\norm c_{\ell^2(\Gamma)}.
\end{equation}
\end{proposition}

\begin{proof}
Put \(p=10/3\), \(p'=10/7\), and
\(Tc(t,\theta)=t^{-1/5}\mathcal U(t)c(\theta)\). By \(TT^*\), it suffices to
prove that \(TT^*:L^{p'}(I_\lambda\times\T)\to
L^p(I_\lambda\times\T)\) is bounded.

For \(t\ne s\), the kernel of \(\mathcal U(t)\mathcal U(s)^*\) is
\[
 K_{t,s}(\omega)
 =\sum_{m\in\Gamma}a(t,m)\overline{a(s,m)}
 e^{i[m\omega+\phi(t,m)-\phi(s,m)]}.
\]
The product rule for discrete variation gives
\[
 \operatorname{Var}_\Gamma(a(t,\cdot)\overline{a(s,\cdot)})
 \le \norm{a(t,\cdot)}_\infty\operatorname{Var}_\Gamma(a(s,\cdot))
 +\norm{a(s,\cdot)}_\infty\operatorname{Var}_\Gamma(a(t,\cdot)),
\]
so the product amplitude has uniformly bounded supremum and variation. Its
support is contained in a consecutive interval with
\(\lesssim1+\lambda^2\min(t,s)\) points. If the common support has at most two
points, then \(|K_{t,s}|\lesssim1\). Otherwise, restrict the sum to its convex
hull. Since the continuous nonvanishing function \(\Phi_{t,s}''\) has constant
sign, Lemma~\ref{lem:discrete-vdc}, with \(\delta=|t-s|\), gives
\[
 |K_{t,s}(\omega)|
 \lesssim
 \bigl(1+\lambda^2\min(t,s)\bigr)|t-s|^{1/2}
 +|t-s|^{-1/2}.
\]
Because \(|t-s|\le1\) and
\[
 \lambda^2\min(t,s)|t-s|\le\lambda^2ts\le1,
\]
we obtain
\begin{equation}\label{eq:weighted-kernel}
 \norm{\mathcal U(t)\mathcal U(s)^*}_{L^1(\T)\to L^\infty(\T)}
 \lesssim|t-s|^{-1/2}.
\end{equation}
The same operator is a Fourier multiplier with uniformly bounded symbol, so
its \(L^2\to L^2\) norm is bounded. Interpolation with
\eqref{eq:weighted-kernel} yields
\begin{equation}\label{eq:osc-pp}
 \norm{\mathcal U(t)\mathcal U(s)^*}_{L^{p'}(\T)\to L^p(\T)}
 \lesssim|t-s|^{-1/5}.
\end{equation}

We also use the scalar estimate
\begin{equation}\label{eq:scalar-weighted-bound}
 \norm{\mathcal K F}_{L^p(0,\infty)}
 \lesssim\norm F_{L^{p'}(0,\infty)},
 \qquad
 \mathcal K F(t)=t^{-1/5}\int_0^\infty
 |t-s|^{-1/5}s^{-1/5}F(s)\dd s.
\end{equation}
To prove it, set
\[
 f(y)=e^{y/p'}F(e^y),\qquad
 h(x)=e^{x/p}\mathcal K F(e^x).
\]
Then \(\norm f_{p'}=\norm F_{p'}\),
\(\norm h_p=\norm{\mathcal K F}_p\), and the substitutions
\(t=e^x\), \(s=e^{x+z}\) give
\[
 h(x)=\int_\mathbb R k(z)f(x+z)\dd z,
 \qquad k(z)=e^{z/10}|1-e^z|^{-1/5}.
\]
The kernel is \(O(|z|^{-1/5})\) near zero and decays exponentially at both
ends, so \(k\in L^{5/3}(\mathbb R)\). Since
\[
 1+\frac1p=\frac1{p'}+\frac1{5/3},
\]
Young's inequality proves \eqref{eq:scalar-weighted-bound}.

Finally, Minkowski's inequality and \eqref{eq:osc-pp} give
\[
 \norm{TT^*H(t,\cdot)}_{L_\theta^p}
 \lesssim t^{-1/5}\int_{I_\lambda}|t-s|^{-1/5}s^{-1/5}
 \norm{H(s,\cdot)}_{L_\theta^{p'}}\dd s.
\]
Apply \eqref{eq:scalar-weighted-bound} after extending
\(s\mapsto\norm{H(s,\cdot)}_{L_\theta^{p'}}\) by zero outside
\(I_\lambda\). This proves the \(TT^*\) bound and hence
\eqref{eq:weighted-sum-bound}.
\end{proof}

\section{The non-glancing  modes}\label{sec:non-glancing}

We now apply the preceding weighted exponential-sum estimate to the non-glancing modes. 

For \(g\in\operatorname{Ran}\Pi_\lambda\), with coefficients as in
\eqref{eq:polar-expansion}, set
\[
 \Omega(u)=\{m\in\mathcal M_N:\nu_m\le\kappa_0^2u\}
\]
and, on \(\mathcal A_\lambda\), define
\begin{equation}\label{eq:ng-definition}
 g_{<}(r(u),\theta)
 =\frac1{\sqrt{2\pi}}\sum_{m\in\Omega(u)}
 c_mR_{N,m}(r(u))e^{im\theta}.
\end{equation}

\begin{proposition}\label{prop:below-threshold}
For every \(g\in\operatorname{Ran}\Pi_\lambda\),
\[
 \norm{g_{<}}_{L^{10/3}(\mathcal A_\lambda)}
 \lesssim\lambda^{-1/10}\norm g_2.
\]
\end{proposition}

\begin{proof}
Put \(p=10/3\). Since \(\nu_m\) is nondecreasing in \(|m|\),
\(\Omega(u)\) is a consecutive symmetric interval in \(\mathcal M_N\). For
\(m\in\Omega(u)\),
\[
 m^2=\frac14+\lambda^2\nu_m(1-\nu_m)
 \le\frac14+\kappa_0^2\lambda^2u.
\]
Hence
\begin{equation}\label{eq:below-threshold-support}
 \#\Omega(u)\lesssim1+\lambda\sqrt u,
 \qquad \Omega(u)\subset\mathcal I(u).
\end{equation}
Indeed, the second assertion follows from
\(1/4\le(\kappa_1^2-\kappa_0^2)\lambda^2u\), which is ensured by the
choice of \(C_*\). Thus Lemma~\ref{lem:lg-radial} applies to all modes in
\eqref{eq:ng-definition}.

Recall that \(u_t=\lambda^2t^2\), and put
\[
 I_\lambda=
 [\sqrt{C_*}\lambda^{-4/3},\sqrt{u_0}\lambda^{-1}]
 \subset(0,\lambda^{-1}).
\]
On this interval, \(\lambda^2t=\lambda\sqrt{u_t}\gtrsim1\). Extend \(d_{m,\sigma}\) by zero for \(m\notin\mathcal I(U_1)\), set
\(c_m^\sigma=d_{m,\sigma}c_m\), and, for \(\sigma=\pm\), define
\[
 \mathcal U_\sigma(t)h(\theta)
 =\sum_{m\in\Omega(u_t)}b(u_t,m)
 e^{i(m\theta+\sigma S(u_t,m))}h_m,
 \qquad t\in I_\lambda.
\]
Then \(\norm{c^\sigma}_2\lesssim\norm c_2\), and
Lemma~\ref{lem:lg-radial} writes the leading part of \(g_<\) as
\[
 \frac1{\sqrt{2\pi}}\lambda^{-1}t^{-1/2}
 \sum_{\sigma=\pm}\mathcal U_\sigma(t)c^\sigma(\theta).
\]

We apply Proposition~\ref{prop:weighted-sum} with
\(\Gamma=\mathcal M_N\) to each \(\mathcal U_\sigma\). Its amplitude is
\(\mathbf1_{\Omega(u_t)}(m)b(u_t,m)\). By
\eqref{eq:below-threshold-support}, its support has
\(O(1+\lambda^2t)\) points, while the two boundary jumps and
\eqref{eq:lg-symbol} give
\[
 \norm{\mathbf1_{\Omega(u_t)}b(u_t,\cdot)}_\infty
 +\operatorname{Var}_{\mathcal M_N}
   (\mathbf1_{\Omega(u_t)}b(u_t,\cdot))
 \lesssim1+\frac{\#\Omega(u_t)}{\lambda^2t}
 \lesssim1.
\]
For the formal definition of the phase on the whole lattice, extend
\(\sigma S(u_t,m)\) by zero when \(m\notin\mathcal I(u_t)\); this does not
change \(\mathcal U_\sigma(t)\), and the resulting functions are measurable.
If \(t\ne s\), then
\[
 \Omega(u_t)\cap\Omega(u_s)=\Omega(\min\{u_t,u_s\}),
\]
whose real convex hull lies in
\(\{|\xi|\le\kappa_1\lambda^2\min(t,s)\}\). On this hull the phase
difference is
\[
 \sigma\bigl(S(u_t,\xi)-S(u_s,\xi)\bigr),
\]
and Lemma~\ref{lem:phase-second-variation} gives the required second
derivative bound and constant sign. Proposition~\ref{prop:weighted-sum}
therefore applies.

The factor \((2\pi)^{-1/2}\) is harmless. Since
\(|r\dd r|=\lambda^3t\dd t\),
\[
 \begin{aligned}
 \norm{\lambda^{-1}t^{-1/2}
 \mathcal U_\sigma(t)c^\sigma}_{L^p(r\dd r\dd\theta)}=\lambda^{-1/10}
 \norm{t^{-1/5}\mathcal U_\sigma(t)c^\sigma}_{L^p(\dd t\dd\theta)}.
 \end{aligned}
\]
Thus \eqref{eq:weighted-sum-bound} bounds the displayed leading term by
\(\lambda^{-1/10}\norm c_2\).

It remains to estimate the sum of the Liouville--Green remainders. Since
\(\lambda\sqrt u\gtrsim1\) on \(\mathcal A_\lambda\), Hausdorff--Young,
H\"older, \eqref{eq:below-threshold-support}, and
\eqref{eq:lg-error},
\[
 \begin{aligned}
 \left\|
 \sum_{m\in\Omega(u)}c_m\mathcal E_m(u)e^{im\theta}
 \right\|_{L_\theta^p}\lesssim
 (\lambda\sqrt u)^{1/5}
 \lambda^{-1/2}u^{-1/4}(\lambda u^{3/2})^{-1}\norm c_2
 =\lambda^{-13/10}u^{-33/20}\norm c_2.
 \end{aligned}
\]
Using \(|r\dd r|=(\lambda/2)\dd u\) and \(u_*=\lambda^{-2/3}\), we obtain
\[
 \begin{aligned}
 \left\|
 \sum_{m\in\Omega(u)}c_m\mathcal E_m(u)e^{im\theta}
 \right\|_{L^p(\mathcal A_\lambda)}^p
\lesssim
 \lambda^{-10/3}\int_{C_*u_*}^{u_0}u^{-11/2}\dd u\,\norm c_2^p
 \lesssim\lambda^{-1/3}\norm c_2^p.
 \end{aligned}
\]
Taking the \(p\)-th root and using \(\norm c_2=\norm g_2\) completes the
proof.
\end{proof}

\section{The remaining  modes}\label{sec:complementary}

This section treats the modes not covered by the non-glancing estimate. We
split them into a component near the radial turning point and a component lying
deep in the forbidden region. The near component is controlled on dyadic
annuli by the localized Koch--Tataru estimates, whereas the far component is controlled by the rapid decay from
Lemma~\ref{lem:forbidden-radial}. 

For \(g\in\operatorname{Ran}\Pi_\lambda\), set
$
 g_{>}=g-g_{<}$ on $\mathcal A_\lambda.$
\begin{proposition}\label{prop:complementary}
For every \(g\in\operatorname{Ran}\Pi_\lambda\), the component \(g_>\)
satisfies
\[
 \norm{g_{>}}_{L^{10/3}(\mathcal A_\lambda)}
 \lesssim\lambda^{-1/10}\norm g_2.
\]
\end{proposition}

\begin{proof}
Put \(p=10/3\). We first split the modes according to their distance from the
turning point:
\[
 \begin{aligned}
 g_{\mathrm{near}}(r(u),\theta)
 &=\frac1{\sqrt{2\pi}}
   \sum_{\kappa_0^2u<\nu_m<u/c_1}
   c_mR_{N,m}(r(u))e^{im\theta},\\
 g_{\mathrm{far}}(r(u),\theta)
 &=\frac1{\sqrt{2\pi}}
   \sum_{u\le c_1\nu_m}
   c_mR_{N,m}(r(u))e^{im\theta}.
 \end{aligned}
\]
Thus \(g_>=g_{\mathrm{near}}+g_{\mathrm{far}}\) on
\(\mathcal A_\lambda\).

We begin with the far part. Since \(u\ge C_*u_*\), every mode occurring in
\(g_{\mathrm{far}}\) satisfies the hypotheses of
Lemma~\ref{lem:forbidden-radial}. Taking \(M=1/2\) gives the particularly
simple bound
\begin{equation}\label{eq:far-mode-decay}
 |R_{N,m}(r(u))|\lesssim\lambda^{-1}\nu_m^{-1}
 \qquad (u\le c_1\nu_m).
\end{equation}
Moreover, \eqref{eq:nu-m-comparable} implies that, for
\(a\ge C_*u_*\),
\begin{equation}\label{eq:nu-tail-sum}
 \sum_{\nu_m\ge a}\nu_m^{-5}
 \lesssim
 \sum_{\ell\ge0}(2^\ell a)^{-5}
 \#\{m:0<\nu_m<2^{\ell+1}a\}
 \lesssim\lambda a^{-9/2},
\end{equation}
where we used \(\#\{m:0<\nu_m<v\}\lesssim\lambda\sqrt v\).
Since \(p'=10/7\) and \(1/p'=1/2+1/5\), Hausdorff--Young, H\"older,
\eqref{eq:far-mode-decay}, and \eqref{eq:nu-tail-sum}, with
\(a=u/c_1\), yield
\[
 \begin{aligned}
 \norm{g_{\mathrm{far}}(r(u),\cdot)}_{L_\theta^p}
 &\lesssim
 \left(\sum_{u\le c_1\nu_m}|c_mR_{N,m}(r(u))|^{p'}\right)^{1/p'}\\
 &\le \norm c_2
 \left(\sum_{u\le c_1\nu_m}|R_{N,m}(r(u))|^5\right)^{1/5}\\
 &\lesssim\lambda^{-4/5}u^{-9/10}\norm g_2.
 \end{aligned}
\]
Using \(|r\dd r|=(\lambda/2)\dd u\) and \(u_*=\lambda^{-2/3}\), we obtain
\[
 \norm{g_{\mathrm{far}}}_p^p
 \lesssim
 \lambda^{-5/3}\int_{C_*u_*}^{u_0}u^{-3}\dd u\,\norm g_2^p
 \lesssim\lambda^{-1/3}\norm g_2^p.
\]
Hence
\begin{equation}\label{eq:far-component-bound}
 \norm{g_{\mathrm{far}}}_p
 \lesssim\lambda^{-1/10}\norm g_2.
\end{equation}

It remains to treat the near part. For \(j\ge0\) with
\(\rho_j=2^jC_*u_*\le u_0\), set
\[
 \mathcal A_j
 =\{x\in\mathcal A_\lambda:\rho_j/2\le u(x)\le2\rho_j\}
\]
and
\[
 G_j(r,\theta)=\frac1{\sqrt{2\pi}}
 \sum_{\kappa_0^2\rho_j/2<\nu_m<2\rho_j/c_1}
 c_mR_{N,m}(r)e^{im\theta}.
\]
The annuli \(\mathcal A_j\) cover \(\mathcal A_\lambda\) with bounded overlap.
At a fixed radius, let \(\mathsf Q_u\) denote the angular Fourier projection
onto
\[
 \{m\in\mathcal M_N:\kappa_0^2u<\nu_m<u/c_1\}.
\]
Since \(\nu_m\) is nondecreasing in \(|m|\), its symbol is the sum of at most
two interval indicators on the step-two lattice \(\mathcal M_N\). The
M.~Riesz theorem
\cite[Proposition~4.1.6 and Theorem~4.1.7]{Grafakos} therefore gives
\[
 \sup_{u>0}\norm{\mathsf Q_u}_{L^p(\T)\to L^p(\T)}\lesssim1.
\]
If \(x\in\mathcal A_j\), every mode occurring in
\(g_{\mathrm{near}}(x)\) satisfies $\kappa_0^2\rho_j/2<\nu_m<2\rho_j/c_1$.
Thus \(g_{\mathrm{near}}=\mathsf Q_uG_j\) on \(\mathcal A_j\), and
\eqref{eq:package-allowed}, applied to
\(G_j\in\operatorname{Ran}\Pi_\lambda\), yields
\begin{equation}\label{eq:near-annular-bound}
 \norm{\one_{\mathcal A_j}g_{\mathrm{near}}}_p
 \lesssim\norm{\one_{\mathcal A_j}G_j}_p
 \lesssim\lambda^{-1/10}\norm{G_j}_2.
\end{equation}
Moreover,
\[
 \norm{G_j}_2^2
 =\sum_{\kappa_0^2\rho_j/2<\nu_m<2\rho_j/c_1}|c_m|^2.
\]
Each \(m\) occurs in only \(O(1)\) of these sums. Hence
\[
 \sum_j\norm{G_j}_2^2\lesssim\norm g_2^2,
 \qquad
 \sum_j\norm{G_j}_2^p
 \le\left(\sum_j\norm{G_j}_2^2\right)^{p/2}
 \lesssim\norm g_2^p,
\]
where the middle inequality uses \(p>2\). Combining this with the bounded
overlap of the annuli and \eqref{eq:near-annular-bound}, we obtain
\begin{equation}\label{eq:near-component-bound}
 \norm{g_{\mathrm{near}}}_p
 \lesssim\lambda^{-1/10}\norm g_2.
\end{equation}
The proposition follows from
\eqref{eq:far-component-bound} and \eqref{eq:near-component-bound}.
\end{proof}

\section{Proof of Theorem~\ref{thm:main}}\label{sec:proof-main}

Let \(g=\Pi_\lambda f\). Then
\(\norm g_2\le\norm f_2\), and let
\(u_{\mathrm{min}}=C_*u_*\).
The deep-interior bound \eqref{eq:package-deep} applies on \(\{u\ge u_0\}\).
Propositions~\ref{prop:below-threshold} and
\ref{prop:complementary}, together with
\(g=g_{<}+g_{>}\), give
\[
 \norm g_{L^{10/3}(\mathcal A_\lambda)}
 \lesssim\lambda^{-1/10}\norm f_2.
\]
Estimate \eqref{eq:package-core}, with \(A=4C_*\), applies on
\(\{|u|\le4u_{\mathrm{min}}\}\). Estimate \eqref{eq:outer-sum}
applies on \(\{u_{\mathrm{min}}\le\eta\le\eta_0\}\), and
\eqref{eq:package-far-outer} applies on
\(\{\eta\ge\eta_0/2\}\).

These five regions cover \(\R^2\). Indeed, an interior point lies in one of
\(\{u\ge u_0\}\), \(\mathcal A_\lambda\), and
\(\{|u|\le4u_{\mathrm{min}}\}\). If an exterior point does not lie in
\(\{|u|\le4u_{\mathrm{min}}\}\), then \(-u>4u_{\mathrm{min}}\). The alternative
\(\eta<u_{\mathrm{min}}\) would give
\[
 -u=2\eta+\eta^2
 <2u_{\mathrm{min}}+u_{\mathrm{min}}^2<4u_{\mathrm{min}},
\]
a contradiction. The point therefore lies either in
\(u_{\mathrm{min}}\le\eta\le\eta_0\) or in \(\eta\ge\eta_0/2\). Thus
the five estimates above prove \eqref{eq:main}.

\end{document}